\documentclass[default,12pt]{sn-jnl}

\usepackage{graphicx}%
\usepackage{multirow}%
\usepackage{amsmath,amssymb,amsfonts}%
\usepackage{amsthm}%
\usepackage{mathrsfs}%
\usepackage[title]{appendix}%
\usepackage{xcolor}%
\usepackage{textcomp}%
\usepackage{manyfoot}%
\usepackage{booktabs}%
\usepackage{algorithm}%
\usepackage{algorithmicx}%
\usepackage{algpseudocode}%
\usepackage{listings}%
\usepackage{amsthm}
\usepackage{amssymb}
\usepackage{physics}
\usepackage[numbers]{natbib}
\usepackage{bm}
\usepackage{esint}
\usepackage{xcolor}
\numberwithin{equation}{section}

\usepackage{soul}

\theoremstyle{thmstyleone}%
\newtheorem{theorem}{Theorem}%
\newtheorem{proposition}[theorem]{Proposition}%
\newtheorem{cor}[theorem]{Corollary}

\theoremstyle{thmstyletwo}%
\newtheorem{remark}{Remark}%

\theoremstyle{thmstylethree}%
\newtheorem{definition}{Definition}%
\newtheorem{notat}{Notation} 
\numberwithin{theorem}{section}
\numberwithin{remark}{section}
\numberwithin{definition}{section}

\begin{document}

\title{Well-Posedness and regularity properties of 2d $\beta$-plane stochastic Navier-Stokes equations in a periodic channel}


\author*[1]{\fnm{Yuri} \sur{Cacchiò }}\email{yuri.cacchio@gssi.it}

\author*[2]{\fnm{Amirali} \sur{Hannani}}\email{amirali.hannani@kuleuven.be}

\author*[3]{\fnm{Gigliola} \sur{Staffilani}}\email{gigliola@math.mit.edu}

\affil[1]{\orgname{Gran Sasso Science Institute}, \orgaddress{\street{Viale Francesco Crispi, 7}, \city{L'Aquila}, \postcode{67100}, \state{Italy}}}

\affil[2]{\orgdiv{Instituut voor Theoretische Fysica}, \orgname{KU Leuven}, \orgaddress{\street{Celestijnenlaan 200d}, \city{Leuven}, \postcode{3001}, \state{Belgium}}}

\affil[3]{\orgdiv{Department of Mathematics}, \orgname{Massachusetts Institue of Technology}, \orgaddress{\street{77 Massachusetts Ave}, \city{Cambridge}, \postcode{02139-4307}, \state{MA}, \country{USA}}}


\abstract{ We consider the 2d $\beta$-plane stochastic Navier-Stokes equations in a periodic  channel. We prove the well-posedness and existence of the stationary measure, as well as certain regularity estimates concerning the support of the stationary
measure.  The mentioned estimates are crucial for the rigorous study of the cascade phenomena in this equation \cite{GigliolaAmiraliYuri}.
To the best of our knowledge, this is the first mathematically rigorous treatment of these equations involving both the stochastic noise and the Coriolis
force.}
\maketitle

\section{Introduction}\label{introduction}
One of the canonical models in fluid dynamics is the \textit{two-dimensional Navier-Stokes equations with additive stochastic noise (2d SNS)}. This equation has been studied extensively both in mathematics and physics communities. We refer the reader to \cite{Bensoussan,daprato,flandoli,Hairer,kuk,Temam,mustafa,Yuri,gallagher2,gallagher1,Shepherd} for a non-exhaustive list of works concerning well-posedness, ergodicity and long time behavior in the stochastic and deterministic case. On the other hand, 
 in the physics literature, we refer to \cite{Batchelor2,Batchelor,boffetta,Cost,Danilov,eyink,Fjrtoft,frisch,huang,Kraichnan,nozawa,rhines,salmon,scott,ScottPolv,sriniv,vallis4,vallis3}, as examples where these equations are used to model various phenomena such as jet streams and 2d turbulence.

Despite the wide range of applications, 2d SNS fails to reflect some features of \textit{geophysical flows}. In fact, if we deal with such flows, we must include the planetary rotation (via Coriolis force) into the stochastic Navier-Stokes equations \cite{gallagher1,holton,McWilliams,pedlosky,vallis1} resulting in the so-called 
$\beta$-plane model \cite{chekhlov,Cost2,Cost,Cope,galerpin2,galerpin1,huang2}.\\

Taking into account the Coriolis force, we consider the following system of equations, which will be the object of study in the present work
\begin{equation}\label{Problem}\left\{
    \begin{array}{rl}
    \partial_t u+\left(u\cdot \nabla\right)u+f u^\perp &=\nu\Delta u-\alpha u-\nabla p+\varphi;\\
    \nabla \cdot u&=0,
\end{array}
\right.
\end{equation}
where $u=(u^1,u^2)$ and $p$ are unknown velocity field and pressure, $(u\cdot\nabla)$ stands for the differential operator $u^1\partial_{x^1}+u^2\partial_{x^2}$ with $x=(x^1,x^2)$, $fu^\perp$ is the Coriolis force where $f$ is the Coriolis parameter (cf. \eqref{Coriolisapprox}) and $u^\perp=(-u^2,u^1)$
, $\varphi$ is the stochastic process, and damping is given by a combination of drag $\alpha>0$ and viscosity $\nu>0$.\\

 We refer the interested reader to \cite{salmon,vallis1} for complete derivations of this system. \\
In this manuscript, we use a well-known approximation of the Coriolis force, the so-called $\beta$-plane approximation \cite{Nazarenko,vallis1}. This regime captures the dynamic effects of rotation. This approximation is done mainly by a Taylor expansion of the Coriolis parameter around a fixed latitude $\varTheta_0$. Consequently, for small variations in latitude, we can write
\begin{equation}
    f=2\Omega\sin\varTheta\approx 2\Omega\sin{\varTheta_0}+2\Omega(\varTheta-\varTheta_0)\cos{\varTheta_0},
\end{equation}
where $\Omega$ is the angular velocity of the sphere. With further approximation, we can assume that the Coriolis parameter varies as
\begin{equation} \label{Coriolisapprox}
    f=f_0+\beta x^2
\end{equation}
for small variation of $x^2$, where $f_0=2\Omega\sin{\varTheta_0}$, $\beta=\partial f/\partial x^2=(2\Omega\cos{\varTheta_0})/\mathfrak{r}$ and $\mathfrak{r}$ is the radius of the planet. Typical values of $\varTheta_0$ are: $0^\circ$ for the so-called equatorial flows, $45^\circ$ for the so-called mid-latitude flows, and around $80^\circ$ for the so-called high latitude flows. 

Introducing the Coriolis force in the $\beta$-plane approximation breaks the symmetry along the $x^2$ direction making the system anisotropic. Hence, we cannot consider a standard double periodic domain (we also refer to \cite{gallagher2,gallagher1} for more discussion in this direction).
This leads us to take the most treatable, physically relevant domain: a periodic channel \cite[p 277]{salmon}.
 Then, we pose the system of equations \eqref{Problem} on a periodic domain in $x^1$, $\mathbb{T}_{x^1}=[0,L)$ torus of size $L$, and a bounded interval in $x^2$, $I=[a,b]$, equipped with periodic boundary conditions in $x^1$ and no-slip boundary conditions in $x^2$, i.e.
\begin{align}
    u(t,0,x^2)&=u(t,L,x^2),\label{periodicboundarycond}\\
    u(t,x^1,\partial I)&=0. \label{noslipboundary}
\end{align}
Finally, we assume that the stochastic process is spatially regular and white  in time as in \cite{kuk}, i.e. 
\begin{equation}\label{stochasticprocess}
    \varphi(t,x)=\frac{\partial}{\partial t}\zeta(t,x),\ \ \ \ \zeta(t,x)=\sum_{j=1}^{\infty} b_j\beta_j(t)e_j(x),\ \ \ \ t\geq0
\end{equation}
where $\{e_j\}$ is a divergence-free orthonormal basis in $H$  (completion of divergence-free smooth functions with proper boundary conditions in $L^2(\mathbb{T}\times I)$,  cf. Definition \ref{spaces definitions}), $b_j$ are  constants such that, 
\begin{align} 
    \varepsilon&:=\frac{1}{2}\sum_{j=1}^\infty b_j^2<\infty, \label{eq: intro: energy}\\
    \eta&:=\frac{1}{2}\sum_j \fint_{\mathbb{T}\cross I}b_j^2|\nabla \cross e_j|^2<\infty, \label{eq: intro: enstrophyavg}
\end{align}
and $\{\beta_j\}$ is a sequence of independent standard Brownian motions (cf. \eqref{def:stochasticforce}). 
The constants $\varepsilon$, $\eta$ are called average energy and enstrophy input per unit time per unit area, respectively. 

To the best of our knowledge, there is no rigorous mathematical study concerning these equations having both the stochastic noise and the Coriolis force.

The goal of this paper is to prove the well-posedness and existence of the stationary measure of \eqref{Problem}. 
By stationary, we mean that the law of $u(\cdot)$ coincides with the law of  $u(\cdot+ \tau)$ for any $\tau \geq 0$.  In addition, some regularity estimates regarding the support of the stationary measure have been proved. More explicitly, we observe that $ \nu \mathbb{E
}(\|\Delta u \|^2_{L^2})$ is bounded, where $\mathbb{E}$ denotes the expectation w.r.t the stationary measure, meaning that the stationary measure is at least supported on $H^2$ functions. Such regularity results pave the way for studying the \textit{cascade} phenomena in a rigorous way as it is motivated by \cite{GigliolaAmiraliYuri}.\\

To obtain mentioned results, we use techniques from
  \cite{Temam} to deal with the peculiarity of the domains, and we take advantage of the ideas presented in \cite{kuk} to deal with the noise. Still, combining ideas presented in \cite{kuk}, \cite{Temam} is not sufficient, and we should deal with the additional difficulties arising from the Coriolis force separately. Let us stress that the above-mentioned results are not straightforward. In fact, by adding the Coriolis force, we have an additional term to control along with the particular choice of the domain. In particular, methods in \cite[p.66]{kuk} and \cite[p.282]{Temam} are not directly applicable.

To be more precise,  similar to \cite{kuk}, the main idea is to split the solution into two sub-problems: the well-posedness of the stochastic Stokes equation and the well-posedness of Navier-Stokes equations with random coefficients. Since the Coriolis force does not affect the stochastic Stokes equation, we can directly apply results from \cite[Proposition 2.4.2]{kuk} with minor adjustments regarding the domain. On the other hand, the Coriolis force appears in the Navier-Stokes equations with random coefficients. Then, we prove the well-posedness of the resulting equation using ideas inspired by \cite[Theorem 3.1]{Temam}. Here the main tool is the Galerkin method. We emphasize that this method is now applied to an equation with random coefficients entering through the solution of the stochastic Stokes equation. Moreover, introducing the Coriolis force leads us to estimate a new term, which, fortunately, we will show that can also be controlled with similar techniques. As mentioned above, the planetary rotation also influences the choice of our domain which is an additional difference with respect to \cite{kuk}, forcing us to use ideas from \cite{Temam}.
Finally, because of the $\beta$-plane approximation, we have different boundary conditions compared to \cite{kuk,Temam}. 

We would like to conclude this part of the introduction by saying that while the readers of this note will learn some generalizations of results in   \cite{kuk,Temam} when the Coriolis force is at play, they will also have a clear summary of well-posedness results for stochastic Navier-Stokes equations in 2D. 

The organization of the paper is as follows. 
In Section \ref{sectionprelimin}, we define a solution of \eqref{Problem} in the sense of Definition \ref{solutiondef}. In Section \ref{Well-Posedness}, we show the well-posedness result (cf. Theorem \ref{TheoSolution}). We prove the existence of a stationary measure for these equations (cf. Theorem \ref{theoremstationarymeasure}) in the last section.  Furthermore, we investigate certain regularity properties of the support of this measure in Theorem \ref{regularity solution thm}.

\bigskip
\textbf{ Acknowledgment} Y.C. was funded in part by Sapienza "Giovani Ricercatori" Grant DR n.3147/2022 and PRIN-MUR grant 2022YXWSLR "Boundary analysis for dispersive and viscous fluids". A.H. was funded in part by the FWO grant G098919N, the ANR grant LSD-15-CE40-0020-01, and the NSF Grant DMS-1929284. G.S. was funded in part by the NSF grant DMS-2052651 and the Simons Foundation through the Simons Collaboration  Wave Turbulence Grant.



\section{Preliminaries and Main Results}\label{sectionprelimin}
The system of equations \eqref{Problem} is posed on a periodic domain in $x^1$, $\mathbb{T}_{x^1}=[0,L)$ torus of size $L>0$, and a bounded domain in $x^2$, $I=[a,b]$, which means $x=(x^1,x^2)\in \mathbb{T}\cross I$.

Let us introduce the following vector spaces and set some notations.
\begin{definition}\label{spaces definitions}
We define 
\begin{itemize}
    \item[1)] $D
    =\{u\in C^\infty(\mathbb{T}\cross (a,b)); u(0,x^2)=u(L,x^2) \text{ with compact support in }x^2\}$
    \item[2)] $D_\sigma
    =\{u\in D(\mathbb{T}\cross (a,b)), \nabla \cdot u=0\}$
    \item[3)] $H^m_0
    =\overline{D(\mathbb{T}\cross (a,b))}^{H^m(\mathbb{T}\cross (a,b))}$
    \item[4)] $H
    =\overline{D_\sigma(\mathbb{T}\cross (a,b))}^{L^2(\mathbb{T}\cross (a,b))}$
    \item[5)] $V
    =\overline{D_\sigma(\mathbb{T}\cross (a,b))}^{H^1_0(\mathbb{T}\cross (a,b))}$
\end{itemize}
where $\overline{D(\mathbb{T}\cross (a,b))}^{H^m(\mathbb{T}\cross (a,b))}$ denotes the closure of $D$ in the Sobolev space $H^m$. A similar notation holds for $H$ and $V$.
In addition, $H$ and $V$ are equipped with the induced norm,
\begin{equation*}
    \norm{\cdot}_{L^2(\mathbb{T}\cross I)}; \ \ \ \norm{\cdot}_{H^1(\mathbb{T}\cross I)},
\end{equation*}
respectively. $H'$ and $V'$ denote the dual space of $H$ and $V$.
\end{definition}

\begin{remark}
By the Riesz representation theorem, we have 
\begin{equation}
    V\subset H=H'\subset V'.
\end{equation}
\end{remark}

\begin{notat}
Throughout the paper we write $L^2$ instead of $L^2(\mathbb{T}\cross I)$ and for the norm we write $\norm{\cdot}_{L^2}$. 
Similarly, for other spaces defined above, from now on, we omit the dependence on the domain $\mathbb{T}\cross I$ whenever it does not cause any confusion.  
\end{notat}

Assuming boundary conditions as in \eqref{periodicboundarycond}, \eqref{noslipboundary} and fixing the initial data 
\begin{equation}
    u(0)=u_0\in H,
\end{equation}
we derive the following Cauchy problem 
\begin{equation}\label{Problem1}\left\{
    \begin{array}{rl}
     \partial_t u+\left(u\cdot \nabla\right)u+fu^\perp &=\nu\Delta u-\alpha u-\nabla p+\varphi;\\ 
     \nabla \cdot u&=0;\\
     u(t,0,x^2)&=u(t,L,x^2);\\
     u(t,x^1,\partial I)&=0;\\
     u(0)&=u_0\in H.
\end{array}
\right.
\end{equation}
We recall that 
\begin{equation} \label{def:stochasticforce}
    \varphi(t,x)=\frac{\partial}{\partial t}\zeta(t,x),\ \ \ \ \zeta(t,x)=\sum_{j=1}^{\infty} b_j\beta_j(t)e_j(x),\ \ \ \ t\geq0
\end{equation}
where $\{e_j\}$ is a divergence-free orthonormal basis in $H$, $b_j$ are  constants satisfying \eqref{eq: intro: energy}-\eqref{eq: intro: enstrophyavg}, and $\{\beta_j\}$ is a sequence of independent standard Brownian motions. 

We assume that the Brownian motions $\{\beta_j\}$ are defined on a complete probability space $(\Omega,\mathcal{F},\mathbb{P})$ with a filtration $\mathcal{G}_t$, $t\geq0$, and the $\sigma-$algebras $\mathcal{G}_t$ are completed with respect to $(\mathcal{F},\mathbb{P})$, that is, $\mathcal{G}_t$ contains all $\mathbb{P}-$null sets $A\in \mathcal{F}$.

Let us note that \eqref{Problem1} is not a closed system of evolution equations in the sense that it does not contain the time derivative of the unknown function $p$. However, we can write an equivalent system without pressure and obtain a nonlinear PDE which can be regarded as an evolution equation in a Hilbert space \cite{kuk,Temam}.

This is done by the following abstract formulation
\begin{equation}\label{Problemleray}\left\{
    \begin{array}{rl}
     \partial_t u+\nu Au+B(u)+\alpha u+Fu&=\varphi; \\
     \nabla \cdot u&=0;\\
     u(t,0,x^2)&=u(t,L,x^2);\\
     u(t,x^1,\partial I)&=0;\\
     u(0)&=u_0\in H 
\end{array}
\right.
\end{equation}
where 
\begin{align}
    A u&=-\Pi\Delta u, \\
    B(u)&=B(u,u)=\Pi((u\cdot\nabla)u),\\
    Fu&=\Pi(fu^\perp),
\end{align}
and
\begin{equation}
    \Pi:H^s(Q,\mathbb{R}^2)\to H^s_\sigma(Q,\mathbb{R}^2), 
\end{equation}
is the orthogonal projection for any bounded Lipschitz domain $Q$ called Leray projection. Here $H^s$ is the standard Sobolev space and $H^s_\sigma$ denotes the $H^s$-divergent free functions. 
The system \eqref{Problemleray}, sometimes called weak formulation of Navier-Stokes equations, was first derived by Leray in \cite{Leray1,Leray3,Leray2}.

Keeping in mind property \eqref{stochasticprocess} of the stochastic process $\varphi$, let us define the solution of system \eqref{Problemleray}.

\begin{definition}\label{solutiondef}
An $H-$valued random process $u(t)$, $t\geq0$, is called a solution for \eqref{Problemleray} if:
\begin{itemize}
    \item [1)] The process $u(t)$ is adapted to the filtration $\mathcal{G}_t$ (cf. \eqref{def:stochasticforce}), and its almost every trajectory belongs to the space
    \begin{equation*}
        \mathcal{X}=C\left(\mathbb{R}_+; H\right)\cap L^2_{loc}\left( \mathbb{R}_+;V\right). 
    \end{equation*}
    \item[2)] Identity \eqref{Problemleray} holds in the sense that, with probability one,
    \begin{equation}\label{solution}
        u(t)+\int_0^t 
        \underbrace{\left(\nu Au+B(u)+\alpha u+Fu\right)}_{f_u(s)}ds=u(0)+\zeta(t), \ \ \ \ t\geq0,
    \end{equation}
    where the equality holds in the space $H^{-1}$.
\end{itemize}
\end{definition}
\begin{remark}
Let us remark that the system of 
equations \eqref{Problem1} and its well-posedness are equivalent to the well-posedness of \eqref{Problemleray} in the sense of Definition \ref{solutiondef}. In fact, there is an equivalence theorem between solutions \cite[pp. 42]{kuk}. This theorem, given a solution of problem \eqref{Problemleray} in the sense of Definition \ref{solutiondef}, allows us to construct $(u,p)$ as solutions of problem \eqref{Problem1}.
\end{remark}

We give now an informal statement of the result here. However, the well-posedness is proved in Theorem \ref{TheoSolution}, the existence of a stationary measure in Theorem \ref{theoremstationarymeasure} and the regularity of its support in Theorem \ref{regularity solution thm}.

\begin{theorem}
For any $\nu,\alpha>0$ and any $\mathcal{G}_0$-measurable random variable $u_0(x) \in H$, problem \eqref{Problemleray} has a unique solution $u(t)$, $t\in [0,T]$ for every fixed $T\in \mathbb{R}_+$, satisfying the initial condition $u(0)=u_0$, almost surely. Moreover, the stochastic Navier-Stokes system \eqref{Problemleray} has at least a stationary measure supported on $H^2$.
\end{theorem}
Let us mention that we have an equivalent result in terms of vorticity 
\begin{equation}
    \omega:=\nabla \cross u=\partial_1 u^2-\partial_2 u^1.
\end{equation}
In terms of vorticity, \eqref{Problem} becomes
\begin{equation}\label{vorticityequation}
    \partial_t \omega+\left(u\cdot \nabla\right)\omega+\beta u^2 =\nu\Delta \omega-\alpha \omega+\nabla \cross \varphi.
\end{equation}
It is not difficult to derive a similar definition of a solution for \eqref{vorticityequation} as in Definition \ref{solutiondef}.
 Then we have the following result:
 
\begin{cor}\label{wellposdvorticity}
For any $\nu,\alpha>0$ and any $\mathcal{G}_0$-measurable random variable $\omega_0(x) \in H$, problem \eqref{vorticityequation} has a unique solution $\omega(t)$, $t\in [0,T]$ for every fixed $T\in \mathbb{R}_+$, satisfying the initial condition $\omega(0)=\omega_0$, almost surely.
\end{cor}
This is a direct consequence of Theorem \ref{TheoSolution}, where we take the curl of the solution.

\section{Well-Posedness}\label{Well-Posedness}

This section is devoted to the proof of the well-posedness of the initial value problem \eqref{Problemleray} in terms of Definition \ref{Problemleray}. 
We follow the footsteps of \cite{kuk}. Here the idea is to study \eqref{Problemleray} by splitting it into stochastic Stokes
equation (cf. Theorem \ref{stokswellp}) and Navier-Stokes equations with random coefficients (cf. Theorem \ref{wellposed}). Since the Coriolis force appears only in the Navier-Stokes equations with random coefficients, we give a detailed proof of its well-posedness using the so-called Galerkin method. But because of the Coriolis force and choice of the boundary conditions,  our equations become intrinsically anisotropic, and well-posedness is not straightforward since we have an additional term to control along with the particular choice of the domain in contrast to \cite{Temam} where the proof is given for the deterministic case. 
Therefore, we revise existing results and redo estimates as in what follows, overcoming additional difficulties. Finally, we obtain the existence and uniqueness of the solution to problem \eqref{Problemleray} by summing the previous results appropriately (cf. Theorem \ref{TheoSolution}).

Then,
we write a general solution of \eqref{Problemleray} as
\begin{equation}
    u=z+v,
\end{equation}
where $z(t)$ is a process that solves the stochastic Stokes system, i.e.,
\begin{equation}\label{ProblemStokes}\left\{
    \begin{array}{rl}
     \partial_t z+\nu A z  &=\varphi(t,x); \\
     z(t,0,x^2)&=z(t,L,x^2);\\
      z(t,x^1,\partial I)&=0;\\
     z(0,x)&=0.
\end{array}
\right.
\end{equation}
and the function $v$ must satisfy the initial value problem
\begin{equation}\label{ProblemDet}\left\{
    \begin{array}{rl}
     \partial_t v+\nu A v+B(v+z)+F(v+z)+\alpha (v+z) &=0; \\
     \nabla \cdot v&=0;\\
     v(t,0,x^2)&=v(t,L,x^2);\\
      v(t,x^1,\partial I)&=0;\\
     v(0,x)&=u_0.
\end{array}
\right.
\end{equation}

The notion of solution for $z(t)$ is similar to $u(t)$ as in Definition \ref{solutiondef}, where one should replace $f_u(s)$ with $\nu Au$ in \eqref{solution}.
We have the following classical results for $z(t)$, we refer reader to \cite[Chapter 2, Proposition 2.4.2]{kuk} for the proof's detail. We should mention that the proof of \cite{kuk} concerns a problem similar to \eqref{ProblemStokes}, but with periodic boundary conditions. However, the exact same proof with minor adjustments applies to our case.

\begin{proposition}\label{stokswellp}
Problem \eqref{ProblemStokes} has a solution $z(\cdot)\in\mathcal{X}$ . Moreover, this solution is unique in the sense that if $\Tilde{z}(t)$ is another solution, then $z\equiv\Tilde{z}$ almost surely.
\end{proposition}
We recall the standard notation 
    \begin{align*}
        \langle u,v \rangle&=\int_{\mathbb{T}\cross I}u\cdot v \ dx.
    \end{align*}
and we use it to summarise some useful properties of $B(\cdot,\cdot)$, and $F$. 
\begin{proposition}\label{properties}
 For any divergence free smooth vector fields, $u$, $v$, $w$ with mentioned boundary conditions (zero in $x^2$, and periodic in $x^1$) for we have 
\begin{align}
    \langle v,Fv \rangle&=0,\label{eq1}\\
    \langle B(u,v),v\rangle&=0,\label{eq2}\\
    \langle B(u,v),w \rangle&=-\langle B(u,w),v \rangle,\label{eq3}\\
    |\langle B(u,v),w \rangle|&\leq C\norm{u}_{H^{1/2}}\norm{v}_{H^{1/2}}\norm{w}_{H^{1}},\label{ineq4}\\
    |\langle B(v,u),v \rangle|&\leq \frac{1}{4}\norm{v}^2_{H^{1}}+C\norm{v}^2_{L^{2}}\norm{u}^2_{H^{1}},\label{ineq5}\\
    \norm{B(u,v)}_{H^{-1}}&\leq C\norm{u}_{H^{1/2}}\norm{v}_{H^{1/2}}
    \label{ineq6}.
\end{align} 
\end{proposition}

\begin{proof}
\begin{itemize}
    \item [1)] By definition 
    $Fv=f(-v^2,v^1)=(f_0+\beta x^2)(-v^2,v^1)$, where $f_0, \beta$ are constants. 
    This means $v$ and $Fv$ are orthogonal and we have $\langle v,Fv \rangle=0$.
    \item[2)] Integrating by parts, and using the boundary conditions mentioned above, we derive 
    \begin{align*}
        \langle B(u,v),v\rangle&=\sum_{j,l=1}^2\int_{\mathbb{T}\times I} u^j(\partial_j v^l)v^l dx=\sum_{j,l=1}^2\frac{1}{2}\int_{\mathbb{T}\times I} u^j\partial_j|v|^2 dx\\&=-\frac{1}{2}\int_{\mathbb{T}\times I} (\nabla \cdot u)|v|^2 dx=0.
    \end{align*}
    \item [3)] By \eqref{eq2}, $\langle B(u,v+w),v+w\rangle=0\Rightarrow \langle B(u,v),w\rangle+\langle B(u,w),v\rangle=0$.
    \item[4)] Applying \eqref{eq3}, Hölder inequality, and the continuous embedding $H^{1/2}\subset L^4$, we obtain 
    \begin{align*}
        |\langle B(u,v),w \rangle|&\leq C_1\int_{\mathbb{T}\times I} |u||v||\nabla w|\leq C_2 \norm{\nabla w}_{L^2}\norm{u}_{L^4}\norm{v}_{L^4}\\
        &\leq C \norm{w}_{H^1}\norm{u}_{H^{1/2}}\norm{v}_{H^{1/2}}.
    \end{align*}
    \item[5)] Using \eqref{ineq4} and Cauchy-Schwarz inequality we have
    \begin{align*}
        |\langle B(v,u),v \rangle|&\leq\norm{u}_{H^{1}}\norm{v}^2_{H^{1/2}}\leq\norm{u}_{H^{1}}\norm{v}_{H^{1}}\norm{v}_{L^{2}}\leq \frac{1}{4}\norm{v}^2_{H^{1}}+C\norm{v}^2_{L^{2}}\norm{u}^2_{H^{1}}.
    \end{align*}
    \item[6)] Follows from \eqref{ineq4} by duality. 
\end{itemize}
Let us emphasize that our choice of boundary conditions do not
affect the argument in the last three estimates.
\end{proof}

\begin{definition}
Let us introduce the function space
\begin{equation}
    \mathcal{H}=\left\{ u\in L^2\left([0,T];V\right);\partial_t u\in L^2\left([0,T];V'\right) \right\}
\end{equation}
and endow it with the norm
\begin{equation}
    \norm{v}_{\mathcal{H}}=\norm{v}_{L^2_{[0,T]}H^1(\mathbb{T}\times I;\mathbb{R}^2)}+\norm{\partial_t v}_{L^2_{[0,T]}H^{-1}(\mathbb{T}\times I;\mathbb{R}^2)}.
\end{equation}
\end{definition}

\begin{proposition}\label{wellposed}
Let $\nu$ and T be some positive constants. Then for any $u_0\in H$ and $z\in \mathcal{X}_T:=C\left([0,T]; H\right)\cap\  L^2\left([0,T];V\right)$ problem \eqref{ProblemDet} has a unique solution $v\in\mathcal{H}$. Moreover, $v\in L^\infty((0,T),H)$ and $v$ is weakly continuous from $[0,T]$ into $H$.
\end{proposition}

\begin{proof}
In this proof, we follow the lines of \cite{Temam}. However, since we have a new term corresponding to the Coriolis force and since the process $z(t)$ (cf. Proposition \ref{stokswellp}) appears in \eqref{ProblemDet}, the proof is not straightforward. In the following, we address these two difficulties. \\
$\mathit{Galerkin\ method:}$ \\
Since $V$ is separable and $D_\sigma$ is dense in $V$, there exists a sequence $\{w_m\}_{m\in \mathbb{N}}\subset D_\sigma$ which is a basis in $V$. For each $m\in \mathbb{N}$, we define an approximate solution $v_m$ of \eqref{ProblemDet} as

\begin{equation}\label{approx solution temam}
    v_m=\sum_{i=1}^m g_{im}(t)w_i,
\end{equation}
such that
\begin{equation}\label{weak approx sol}
    \langle\partial_t v_m,w_j\rangle+ \nu\langle \nabla v_m,\nabla w_j\rangle+ \langle B(v_m+z),w_j\rangle+ \langle F (v_m+z),w_j\rangle+\langle\alpha (v_m+z),w_j\rangle=0,
\end{equation}
with
\begin{equation} \label{weak aprox in.data}
    v_m(0)=v_{0m}
\end{equation}
for $t\in[0,T]$ and $j=1,...,m$, where $v_{0m}$ is the orthogonal projection in $H$ of $u_0$ onto the space spanned by $w_1,...,w_m$.

Equations \eqref{weak approx sol}-\eqref{weak aprox in.data} form a nonlinear differential system for the functions $g_{1m},...,g_{mm}$. In fact, using \eqref{approx solution temam} we rewrite \eqref{weak approx sol} as follows 
\begin{align*}
    &\sum_{i=1}^m\langle w_i,w_j\rangle\partial_t g_{im}+\nu\sum_{i=1}^m\langle\nabla w_i,\nabla w_j\rangle g_{im}+ \sum_{i,l=1}^m\langle B(w_i,w_l),w_j\rangle g_{im}g_{lm}\\
    +&\sum_{l=1}^m\langle B(z,w_l),w_j\rangle g_{lm}
    +\sum_{i=1}^m\langle B(w_i,z),w_j\rangle g_{im}
    +\langle B(z,z),w_j\rangle
    + \sum_{i=1}^m\langle F w_i,w_j\rangle g_{im}
    \\+& \langle F z,w_j\rangle
+\alpha\sum_{i=1}^m\langle w_i,w_j\rangle g_{im}
    +\alpha\langle z,w_j\rangle=0.
\end{align*}
Inverting the non-singular matrix with elements $\langle w_i,w_j\rangle$, $1\leq i,j\leq m$, we can write the differential equations in the usual form
\begin{align}\label{longequationtemam}
    &\partial_t g_{im}(t)+\sum_{j=1}^m c^1_{ij} g_{jm}(t)+ \sum_{j,k=1}^m c^2_{ijk} g_{jm}(t)g_{km}(t)+\sum_{j=1}^m c^3_{ij}(t)  g_{jm}(t)
    +\sum_{j=1}^m c^4_{ij}(t) g_{jm}(t)\notag\\
    &+\sum_{j=1}^m c^5_{ij}\langle B(z,z),w_j\rangle
    + \sum_{j=1}^m c^6_{ij} g_{jm}(t)
    +\sum_{j=1}^m c^7_{ij} \langle F z,w_j\rangle
    +\sum_{j=1}^m c^8_{ij} g_{jm}(t)\\
    &+\sum_{j=1}^m c^9_{ij}\langle z,w_j\rangle=0,\notag
\end{align}
where $c^n_\star$ are constants for $n\neq 3,4$. We have non-constant coefficients in $c^3_{ij}$ and $c^4_{ij}$ since $z$ is time dependent. However, $z(t)$ is continuous so we can proceed as follows.
We rewrite \eqref{longequationtemam} as,
\begin{align}\label{nonlinear system temam}
    &\partial_t g_{im}(t)+\sum_{j=1}^m \alpha^1_{ij}(t) g_{jm}(t)+ \sum_{j,k=1}^m c^2_{ijk} g_{jm}(t)g_{km}(t)
    +\sum_{j=1}^m \alpha^2_{ij}\langle h(t),w_j\rangle=0
\end{align}
where 
\begin{equation}\label{function h}
    h(t)=B(z(t),z(t))+Fz(t)+z(t).
\end{equation}
The condition $v_m(0)=v_{0m}$ is equivalent to $m$ scalar initial conditions
\begin{equation}\label{initial condition temam}
    g_{im}(0)=v_{0m}^i, \text{ i-th component }.
\end{equation}

The nonlinear differential system \eqref{nonlinear system temam} with the initial condition \eqref{initial condition temam} has a maximal solution defined on some interval $[0,t_m]$. If $t_m<T$, then $|v_m(t)|$ must tend to $\infty$ as $t\to t_m$. The following a priori estimates show that this does not happen and therefore $t_m=T$.\\
$\mathit{A\ priori\ estimate:}$\\
We multiply \eqref{weak approx sol} by $g_{jm}(t)$ and add these equations for $j=1,...,m$.
By Proposition \ref{properties} we get
\begin{equation}
    \langle\partial_t v_m,v_m\rangle+ \nu\norm{\nabla v_m}^2_{L^2}+ \langle B(v_m,z),v_m\rangle+\langle B(z,z),v_m\rangle+ \langle F z,v_m\rangle+\langle\alpha (v_m+z),v_m\rangle=0.
\end{equation}
Using notation \eqref{function h}, we derive
\begin{align*}
    &\frac{1}{2}\frac{\partial}{\partial t}\norm{v_m(t)}^2_{L^2}=\langle v_m(t),\partial_t v_m(t)\rangle\\
    =&-\norm{\nabla v_m}^2_{L^2}-\langle v_m,B(v_m,z)\rangle-\langle v_m,B(z,z)\rangle-\langle v_m,Fz\rangle-\alpha\langle v_m,v_m\rangle-\alpha\langle v_m,z\rangle \\
    =&-\norm{\nabla v_m}^2_{L^2}+\langle v_m,B(v_m,z)\rangle-\alpha\langle v_m,v_m\rangle-\langle v_m,Fz+B(z,z)+ z \rangle\\
    \leq& -\norm{\nabla v_m}^2_{L^2}+\frac{1}{4}\norm{ v_m}^2_{H^1}-\alpha\norm{v_m}^2_{L^2}+c\left(\norm{v_m}^2_{L^2}\norm{z}^2_{H^1}+\norm{h}^2_{H^{-1}}\right)\\
    \leq& -\norm{\nabla v_m}^2_{L^2}+\frac{1}{2}\left(\norm{v_m}^2_{L^2}+\norm{\nabla v_m}^2_{L^2}\right)-\alpha\norm{v_m}^2_{L^2}+c\left(\norm{v_m}^2_{L^2}\norm{z}^2_{H^1}+\norm{h}^2_{H^{-1}}\right).
\end{align*}
So that we obtain,
\begin{equation*}
    \frac{\partial}{\partial t}\norm{v_m(t)}^2_{L^2}+\norm{\nabla v_m(t)}^2_{L^2}+\alpha\norm{v_m}^2_{L^2}\leq C\norm{v_m(t)}^2_{L^2}(1+\norm{z}^2_{H^1})+C\norm{h}^2_{H^{-1}}.
\end{equation*}
Since $\alpha\norm{v_m}^2_{L^2}\geq 0$, we have the following inequality:
\begin{equation}\label{comp0}
    \frac{\partial}{\partial t}\norm{v_m(t)}^2_{L^2}+\norm{\nabla v_m(t)}^2_{L^2}\leq C\norm{v_m(t)}^2_{L^2}(1+\norm{z}^2_{H^1})+C\norm{h}^2_{H^{-1}}.
\end{equation}
Furthermore, we observe the following bounds: 
\begin{itemize}
    \item 
    Bounding $H^{-1}$, by 
    $L^2$, using the fact that $(f_0+ \beta x^2)$ is uniformly bounded on $\mathbb{T} \times I$, and taking advantage of the fact that rotation does not change the $L^2$ norm, we get:
    \begin{align}\label{comp1}
\norm{Fz}_{H^{-1}}&\leq\norm{Fz}_{L^2}\leq\norm{(f_0+\beta x^2) \Gamma z}_{L^2}\leq c\norm{z}_{L^2}\leq c\norm{z}_{L^\infty_T L^2}\leq C_T
    \end{align}
    \item By the definition of $B$, separating $\langle k \rangle \cdot 1/\langle k \rangle$ and a Cauchy-Schwartz inequality, we get:
    \begin{align}\label{comp2}
        \norm{B(z,z)}_{H^{-1}}&\leq c\norm{z}^2_{H^{1/2}}\leq c\norm{z}_{L^2}\norm{z}_{H^1}\leq c(\norm{z}^2_{L^2}+\norm{z}^2_{H^1})\notag\\ &\leq c(\norm{z}^2_{L^\infty_T L^2}+\norm{z}^2_{H^1}) \leq C_T(1+\norm{z}^2_{H^1})
    \end{align} 
    \item Thanks to the fact that $z \in 
    \mathcal{X}=C([0,T];L^2_{\sigma})$
    \begin{align}\label{comp3}
        \norm{\alpha z}_{H^{-1}}\leq\alpha\norm{z}_{L^2}\leq\alpha\norm{z}_{L^\infty_T L^2}= C.
    \end{align} 
\end{itemize}
By definition \eqref{function h}, due to \eqref{comp0}, \eqref{comp1}, \eqref{comp2},
and \eqref{comp3} we have
\begin{equation}\label{ineq}
    \frac{\partial}{\partial t}\norm{v_m(t)}^2_{L^2}+\norm{\nabla v_m(t)}^2_{L^2}\leq C\norm{v_m(t)}^2_{L^2}(1+\norm{z}^2_{H^1})+C_T(1+\norm{z}^2_{H^1}).
\end{equation}
Let us consider, 
\begin{equation*}
    \frac{\partial}{\partial t}\norm{v_m(t)}^2_{L^2}\leq C\norm{v_m(t)}^2_{L^2}(1+\norm{z}^2_{H^1})+C_T(1+\norm{z}^2_{H^1}).
\end{equation*}
Using Gronwall's inequality, 
\begin{align*}
 \norm{v_m}^2_{L^2}\leq&\norm{v_{0m}}^2_{L^2}e^{C\int_0^t(1+\norm{z}^2_{H^1})ds}\\
+&C\int_0^t(1+\norm{z}^2_{H^1})e^{C(\int_0^t(1+\norm{z}^2_{H^1})ds-\int_0^s(1+\norm{z}^2_{H^1})d\alpha)}ds,
\end{align*}
since $z\in\mathcal{X}$ it follows that,
\begin{equation}\label{bound}
    \norm{v_m(t)}^2_{L^2}\leq C_T.
\end{equation}
After an integration in $[0,T]$ we get
\begin{equation}
    \norm{v_m(t)}_{L^2_T L^2}\leq C_T.
\end{equation}
Moreover, by \eqref{ineq}, we have
\begin{equation*}
    \norm{\nabla v_m(t)}^2_{L^2}\leq C\norm{v_m(t)}^2_{L^2}(1+\norm{z}^2_{H^1})+C_T(1+\norm{z}^2_{H^1}).
\end{equation*}
Thanks to \eqref{bound} and integrating in $[0,T]$ we obtain 
\begin{equation}\label{boundgradient}
    \norm{\nabla v_m(t)}_{L^2_T L^2}\leq C_T.
\end{equation}
We have thus shown that 
\begin{equation}\label{boundL2H1}
    \norm{v_m}_{L^2_tH^1}\leq C_T.
\end{equation}
The above estimates imply that,
\begin{equation}\label{temam linf bound}
    \text{The sequence $v_m$ remains in a bounded set of }L^\infty([0,T], H).
\end{equation}
and
\begin{equation}
    \text{The sequence $v_m$ remains in a bounded set of }L^2([0,T], V).\label{Temam l2 bound}
\end{equation}
$\mathit{Passage\ to\ the\ limit.}$\\
Let 
\begin{equation}
    \Tilde{v}_m:=\left\{\begin{array}{ll}
         v_m &\text{ on } [0,T] \\
          0 &\text{ on } \mathbb{R}\backslash[0,T].
    \end{array}\right.
\end{equation}
The Fourier transform (in time) of $\Tilde{v}_m$ is denoted by $\hat{v}_m$.\\
We want to show that 
\begin{equation}
    \int_{-\infty}^{+\infty}|\tau|^{2\gamma}|\hat{v}_m(\tau)|^2 d\tau\leq C, \text{ for some }\gamma>0.
\end{equation}
Along with \eqref{Temam l2 bound}, this will imply that 
\begin{equation}\label{temam bounded sequence fourier}
    \Tilde{v}_m \text{ belongs to a bounded set of } \mathcal{H}^\gamma(\mathbb{R}, V,H)
\end{equation}
where 
\begin{equation*}
    \mathcal{H}^\gamma(\mathbb{R}, V,H)=\{v\in L^2(\mathbb{R},V);d^\gamma_t v\in L^2(\mathbb{R},H)\}
\end{equation*}
so that we can use the compactness result in \cite[Theorem 2.2]{Temam}.

We observe that \eqref{weak approx sol} can be written 
\begin{equation}\label{extended equation for u}
    d_t \langle \Tilde{v}_m,w_j\rangle=\langle \Tilde{f}_m,w_j\rangle+\langle v_{0m},w_j\rangle\delta_0-\langle v_m(T),w_j\rangle\delta_T, \ \ j=1,...,m
\end{equation}
where $\delta_0$ and $\delta_T$ are Dirac distributions at $0$ and $T$,
\begin{align}
    f_m:=-\nu Av_m-B(v_m+z)-F(v_m+z)-\alpha(v_m+z)
\end{align}
and 
\begin{equation}
    \Tilde{f}_m:=\left\{\begin{array}{ll}
         f_m &\text{ on } [0,T] \\
          0 &\text{ on } \mathbb{R}\backslash[0,T].
    \end{array}\right.
\end{equation}
By Fourier transform \eqref{extended equation for u}, multiply by $\hat{g}_{jm}(\tau)$ (Fourier transform of $\Tilde{g}_{jm}$ ) and add the resulting equations for $j=1,...,m$, we get 
\begin{equation}\label{fourier relation temam}
    2i\pi\norm{\hat{v}_m(\tau)}^2_{L^2}=\langle \hat{f}_m(\tau),\hat{v}_m(\tau)\rangle+\langle v_{0m},\hat{v}_m(\tau)\rangle-\langle v_m(T),\hat{v}_m(\tau)\rangle e^{-2i\pi T\tau}.
\end{equation}
First, using same techniques as in \eqref{comp1},\eqref{comp2} and \eqref{comp3}, i.e., bounding $H^{-1}$ by $L^2$, using the fact that $(f_0+ \beta)y$ is uniformly bounded on $\mathbb{T} \times I$, taking advantage of the fact that rotation does not change the $L^2$ norm, we derive
\begin{equation*}
    \int_0^T \norm{f_m(t)}_{H^{-1}}\leq C_T.
\end{equation*}
In fact,
\begin{itemize}
    \item $\int_0^T\norm{\Delta v_m}^2_{H^{-1}}dt
    \leq c\int_0^T\norm{v_m}^2_{H^1}dt\leq C_T$,
    \item $\int_0^T\norm{B(v_m,v_m)}^2_{H^{-1}}dt\leq c\int_0^T\norm{v_m}^2_{L^2}\norm{v_m}^2_{H^1}dt\leq C_T\int_0^T\norm{v_m}^2_{H^1}dt\leq C_T$,
    \item $\int_0^T\norm{B(z,z)}^2_{H^{-1}}dt\leq c\int_0^T\norm{z}^2_{L^2}\norm{z}^2_{H^1}dt\leq c\int_0^T\norm{z}^2_{L^\infty_T L^2}\norm{z}^2_{H^1}dt\leq C_T$,
    \item $\int_0^T\norm{B(z,v_m)}^2_{H^{-1}}dt\leq c\int_0^T\norm{v_m}^2_{H^{1/2}}\norm{ z}^2_{H^{1/2}}dt\leq C_T\int_0^T\norm{v_m}_{H^{1}}\norm{z}_{H^{1}}dt\\
    \leq C_T\norm{v_m}_{L^2_TH^{1}}\norm{z}_{L^2_TH^{1}}\leq C_T$,
    \item $\int_0^T\norm{B(v_m,z)}^2_{H^{-1}}dt\leq c\int_0^T\norm{v_m}^2_{H^{1/2}}\norm{ z}^2_{H^{1/2}}dt\leq C_T$,
    \item $\int_0^T\norm{Fv_m}^2_{H^{-1}}dt\leq c\int_0^T\norm{v_m}^2_{L^2}dt\leq C_T$,
    \item $\int_0^T\norm{Fz}^2_{H^{-1}}dt\leq c\int_0^T\norm{z}^2_{L^2}dt\leq C_T$,
    \item $\int_0^T\norm{\alpha v_m}^2_{H^{-1}}dt\leq\alpha\int_0^T\norm{ v_m}^2_{L^2}dt\leq C_T$
    \item $\int_0^T\norm{\alpha z}^2_{H^{-1}}dt\leq\alpha\int_0^T\norm{ z}^2_{L^2}dt\leq C_T$.
\end{itemize}
according to \eqref{Temam l2 bound} and assumptions on $z$. Therefore,
\begin{equation*}
    \sup_{\tau\in\mathbb{R}}\norm{\hat{f}_m(\tau)}_{H^{-1}}\leq C\ \ \text{for all }m\in\mathbb{N}.
\end{equation*}
Due to \eqref{bound}
\begin{equation*}
    \norm{v_m(0)}_{L^2}\leq C\ \text{ and }\ \norm{v_m(T)}_{L^2}\leq C.
\end{equation*}
Still separating $\langle k \rangle \cdot 1/\langle k \rangle$ and using Cauchy-Schwartz inequality, we deduce from \eqref{fourier relation temam} that 
\begin{equation}
    |\tau|\norm{\hat{v}_m(\tau)}^2_{L^2}\leq C\norm{\hat{v}(\tau)}_{H^1}.
\end{equation}
For $\gamma$ fixed, $\gamma<1/4$, we observe that 
\begin{equation}
    |\tau|^{2\gamma}\leq C(\gamma)\frac{1+|\tau|}{1+|\tau|^{1-2\gamma}},\ \ \forall \tau\in\mathbb{R}.
\end{equation}
Thus,
\begin{equation*}
    \int_{-\infty}^{+\infty}|\tau|^{2\gamma}\norm{\hat{v}_m(\tau)}^2_{L^2} d\tau\leq C_1 \int_{-\infty}^{+\infty}\frac{\norm{\hat{v}_m(\tau)}_{H^1}}{1+|\tau|^{1-2\gamma}} d\tau+C_2 \int_{-\infty}^{+\infty}\norm{\hat{v}_m(\tau)}^2_{H^1} d\tau
\end{equation*}
By Parseval equality and \eqref{Temam l2 bound} the last integral is bounded as $m\to\infty$. On the other hand we bound the first integral by Cauchy-Schwartz inequality by the following term
\begin{equation*}
    \left(\int_{-\infty}^{+\infty}\frac{1}{(1+|\tau|^{1-2\gamma})^2}d\tau\right)^{1/2}\left(\int_0^T\norm{v_m(t)}^2_{H^1}dt\right)^{1/2}
\end{equation*}
which is finite since $\gamma<1/4$, and bounded as $m\to\infty$ by \eqref{Temam l2 bound}.

Now we can directly use the argument in \cite{Temam}. We report the main steps here for the reader's convenience.

The estimates \eqref{temam linf bound} and \eqref{Temam l2 bound} enable us to assert the existence of an element $u\in L^2([0,T],V)\cap L^\infty ([0,T], H)$ and a sub-sequence $v_{m'}$ such that 
\begin{equation}
    v_{m'}\to v \text{ in } L^2([0,T],V) \text{ weakly }
\end{equation}
and  
\begin{equation}
    v_{m'}\to v \text{ in } L^\infty([0,T],H) \text{ weakly-star}
\end{equation}
as $m'\to\infty$.

Due \eqref{temam bounded sequence fourier} and \cite[Theorem 2.2]{Temam}, we also have 
\begin{equation}
    v_{m'}\to v \text{ in }L^2([0,T],H) \text{ strongly}. 
\end{equation}
Let $\psi$ be a continuous differentiable function on $[0,T]$ with $\psi(T)=0$. We multiply \eqref{weak approx sol} by $\psi(t)$ and integrate by parts in $t$  the first term 
\begin{align}
    &-\int_0^T \langle v_{m'},w_j\partial_t\psi(t)\rangle+ \nu\int_0^T\langle \nabla v_{m'},\nabla w_j\psi(t)\rangle dt+ \int_0^T\langle B(v_{m'}+z),w_j\psi(t)\rangle dt\notag\\
    &+ \int_0^T\langle F (v_{m'}+z),w_j\psi(t)\rangle dt+\int_0^T\langle\alpha (v_{m'}+z),w_j\psi(t)\rangle dt=\langle v_{0m'},w_j\rangle\psi(0),\label{limit equality temam}
\end{align}
We can pass to limit for $m'\to\infty$ by dominated convergence theorem. For the nonlinear term, we apply \cite[Lemma 3.2]{Temam}. 

Since \eqref{limit equality temam} holds for general $w_j$, by linearity and continuity 
argument 
we can deduce that the same relation holds for any $g\in V$. Now writing in particular for $\psi=\phi\in C^\infty_c((0,T))$ we see that $v$ satisfies \eqref{ProblemDet} in distribution sense. 

Finally, it remains to prove that $v$ satisfies the initial condition $u_0$. As mentioned in \cite{Temam}, we have an equivalent weak formulation of \eqref{ProblemDet}, namely, 
\begin{align*}
    & \partial_t\langle v,g\rangle+ \nu\langle \nabla v,\nabla g\rangle + \langle B(v+z),g\rangle+\langle F (v+z),g\rangle+\langle\alpha (v+z),g\rangle =0
\end{align*}
for any $g\in V$.
Multiplying by $\psi$ and integrating, we derive 
\begin{align*}
    &-\int_0^T \langle v,g\partial_t\psi(t)\rangle+ \nu\int_0^T\langle \nabla v,\nabla g\psi(t)\rangle dt+ \int_0^T\langle B(v+z),g\psi(t)\rangle dt\\
    &+ \int_0^T\langle F (v+z),g\psi(t)\rangle dt+\int_0^T\langle\alpha (v+z),g\psi(t)\rangle dt=\langle v(0),g\rangle\psi(0)
\end{align*}
for any $g\in V$. By comparison with \eqref{limit equality temam}, 
\begin{equation*}
    \langle v(0)-u_0, g\rangle\psi(0)=0.
\end{equation*}
We can chose $\psi$ with $\psi(0)=1$ so that 
\begin{equation*}
    \langle v(0)-u_0, g\rangle=0 \ \ \forall g\in V.
\end{equation*}
This imply $v(0)=u_0$.
\end{proof}

The solution $v$ of $\eqref{ProblemDet}$ satisfies some further regularity properties as shown in the following theorem.
\begin{theorem}
    The solution $v$ of problem \eqref{ProblemDet} given by Proposition \ref{wellposed} is unique. Moreover $v$ is almost everywhere equal to a function continuous from $[0,T]$ into $H$ and
    \begin{equation}
        v(t)\to u_0, \text{ in }H,\text{ as } t\to 0.
    \end{equation}
\end{theorem}
\begin{proof}
Since we proved the same regularity properties of $v$ as in \cite{Temam}, the statement of the Theorem is straightforward from \cite[Theorem 3.2]{Temam}. Note that here, boundary conditions and the new term do not change the classical argument which involves a growth bound on the difference of the solutions and a Gronwall argument concluding the uniqueness.
\end{proof}

 Combining the previous results we derive the following theorem.

\begin{theorem}\label{TheoSolution}
For any $\nu,\alpha>0$ and any $\mathcal{G}_0$-measurable random variable $u_0(x) \in H$, problem \eqref{Problemleray} has a unique solution $u(t)$, $t\in [0,T]$ for every fixed $T\in \mathbb{R}_+$, satisfying the initial condition $u(0)=u_0$, almost surely. Furthermore, the solution $u(t)$ possesses the following properties.
\begin{itemize}
    \item Almost all trajectories of $u(t)$ are continuous with range in $H$ and locally square integrable with range in $V$.
    \item The process $u(t)$ can be written in the form
    \begin{equation}\label{property2}
        u(t)=u_0+\int_0^t f(s) ds+\zeta(t), \ \ \ t\geq0,
    \end{equation}
    where $f(t):=\nu\Delta u(t)-\alpha u(t)-B(u,u)-Fu(t)$ is a $V'$-valued $\mathcal{G}_t$-progressively measurable process such that 
    \begin{equation}\label{propertyprob}
        \mathbb{P}\left\{ \int_0^T\norm{f(t)}^2_{V'}dt<\infty \text{ for any } T>0\right\}=1.
    \end{equation}
\end{itemize}
\end{theorem}

\begin{proof}
We first prove the uniqueness.
 If $u$ and $\Tilde{u}$ are two solutions, then for almost every $\omega\in\Omega$ the difference $w=u-\Tilde{u}$ belongs to $\mathcal{X}$, vanishes at $t=0$ and satisfies the equation
\begin{equation*}
    \partial_t w+B(w,u)+B(\tilde{u},w)+F w+\alpha w+\nu A w=0.
\end{equation*}
Then, multiplying by $w$, integrating and following the same steps as in the proof of Proposition \ref{wellposed}, it follows that $w=0$ almost surely.

To prove the existence, we denote by $z(t)$ the solution of \eqref{ProblemStokes} vanishing at zero. Let $\Omega_0\subset\Omega$ be a set of full measure that consists of those $\omega\in\Omega$ for which $z\in \mathcal{X}$. We define $v\in \mathcal{X}$ as the solution of \eqref{ProblemDet} for $\omega\in\Omega_0$ and set $v=0$ on the complement of $\Omega_0$. Then the random process $u=v+z$ is a solution of \eqref{Problemleray}.

We now prove the two properties. The continuity of trajectories of $u$ in $L^2_\sigma$ follows for a similar property for $z$ and $v$.  Relation \eqref{solution} implies that $u(t)$ can be written in the form \eqref{property2} with $f(t)=-\nu A u(t)-\alpha u(t)-B(u,u)-Fu(t)$. Finally, \eqref{propertyprob} follows from Proposition \ref{stokswellp} and \ref{wellposed}. In fact, due to the computations made in the poof of Proposition \ref{wellposed}, Step 1, it is sufficient to show
\begin{equation*}
    \int_0^T\norm{\Delta z}^2_{H^{-1}}dt
    \leq c\int_0^T\norm{z}^2_{H^1}dt\leq C_T
\end{equation*}
where the last estimate follows from Proposition \ref{stokswellp}.
\end{proof}

\section{Existence of a Stationary Distribution}\label{stationarysolution}

Let us start by deriving an energy balance relation for problem \eqref{Problem}. The latter refers to an equation that describes how the energy is transferred in the fluid \cite{bed2D,kuk}. Moreover, this relation provides estimates on $u$ and $\nabla u$ (cf. \eqref{H1bound}), which are crucial to showing the existence of a stationary measure. 
As we see below, the energy balance equation derived here is identical to the energy balance in the absence of the Coriolis force. This is because the Coriolis force vanishes due to the orthogonality property \eqref{eq1}.

\begin{proposition}\label{propenergybalance}
For any $\nu,\alpha>0$ and any $\mathcal{G}_0$-measurable random variable $u_0(x) \in H$ such that $\mathbb{E}\norm{u_0}^2_{L^2}<\infty$, the following relation holds for a solution $u(t)$ of problem \eqref{Problem1},
\begin{align}\label{equality}
    \mathbb{E}\norm{u(t)}^2_{L^2}&+2\alpha\mathbb{E}\int_0^{t} \norm{u(s)}_{L^2}^2 ds +2\nu\mathbb{E}\int_0^{t} \norm{\nabla u(s)}_{L^2}^2 ds =\mathbb{E}\norm{u_0}^2_{L^2}+2\varepsilon t, 
\end{align}
for $t\geq0$, where $\mathbb{E}$ denotes the expectation w.r.t the law of  the process $u$. 
\end{proposition}

\begin{proof}
We use the It\^{o} formula in Hilbert space (cf. \cite[Theorem 7.7.5]{kuk}) to the functional $\Phi(u)=\norm{u}^2_{L^2}$.
Since
\begin{equation*}
    \partial_u\Phi(u;v)=2\langle u,v\rangle, \ \ \ \partial_u^2\Phi(u;v)=2\norm{v}^2_{L^2},
\end{equation*}
with an application of Cauchy-Schwartz and the dominated convergence theorem we see that the assumptions of the It\^{o} formula are satisfied.

Then, we have 
\begin{align}\label{eqinter}
    \norm{u(t\wedge\tau_n)}^2_{L^2}=\norm{u_0}^2_{L^2}&+2\int_0^{t\wedge\tau_n} \langle u,\nu\Delta u-\alpha u\rangle+\varepsilon\ ds\\
    &+2\sum_{j=1}^\infty b_j\underbrace{\int_0^{t\wedge\tau_n} \langle u,e_j\rangle d\beta_j(s)}_{=:M(t\wedge\tau_n)},\notag
\end{align}
where we set 
\begin{equation*}
    \tau_n=\inf\{t\geq0;\ \norm{u(t)}_{L^2}>n\}.
\end{equation*}
The nonlinear term and pressure vanish after integration by parts as follows
\begin{align}
    \langle (u\cdot \nabla)u,u\rangle&=\sum_{j,l=1}^2\int_{\mathbb{T}\times I} u^j(\partial_j u^l)u^l dx=\sum_{j=1}^2\frac{1}{2}\int_{\mathbb{T}\times I} u^j\partial_j|u|^2 dx\notag\\
        &=-\frac{1}{2}\int_{\mathbb{T}\times I} (\nabla \cdot u)|u|^2 dx=0\label{severalintegration2}\\
    \langle \nabla p, u\rangle&=\sum_{j=1}^2\int_{\mathbb{T}\cross I}u^j\partial_j p dx=-\sum_{j=1}^2\int_{\mathbb{T}\cross I}p\partial_j u^j dx=-\langle p,\nabla \cdot u\rangle=0.
\end{align}
On the other hand $\langle u,f\hat{k}\cross u\rangle=f\langle(u^1,u^2),(-u^2,u^1)\rangle=0$ by orthogonality.

Taking the mean value in \eqref{eqinter} we obtain
\begin{align}
     \mathbb{E}\norm{u(t\wedge\tau_n)}^2_{L^2}&+2\alpha\mathbb{E}\int_0^{t\wedge\tau_n} \norm{u}^2_{L^2}\ ds+2\nu\mathbb{E}\int_0^{t\wedge\tau_n} \norm{\nabla u}^2_{L^2}\ ds\notag\\
     &=\mathbb{E}\norm{u_0}^2_{L^2}+2\varepsilon\mathbb{E}(t\wedge\tau_n)+2\mathbb{E}\sum_{j=1}^\infty b_j M(t\wedge\tau_n)
\end{align}
We remark that, since $\beta_j$ is a sequence of Brownian motion, the last term on the right hand side is a martingale stochastic process. Then we can use Doob's optional sampling Theorem \cite{grimmett} to show 
\begin{equation*}
    \mathbb{E}\sum_{j=1}^\infty b_j M(t\wedge\tau_n)=\mathbb{E}\sum_{j=1}^\infty b_j M(0)=0.
\end{equation*}
It follows, 
\begin{align}\label{eqalim}
    \mathbb{E}\norm{u(t\wedge\tau_n)}^2_{L^2}+2\alpha\mathbb{E}\int_0^{t\wedge\tau_n} \norm{u}^2_{L^2}\ ds&+2\nu\mathbb{E}\int_0^{t\wedge\tau_n} \norm{\nabla u}^2_{L^2}\ ds\\&=\mathbb{E}\norm{u_0}^2_{L^2}+2\varepsilon\mathbb{E}(t\wedge\tau_n).\notag
\end{align}
Since the trajectories of $u(t)$ are continuous $H-$valued functions of time, we have $\tau_n\to\infty$ as $n\to\infty$. Passing to the limit in \eqref{eqalim} as $n\to\infty$ and using the monotone convergence theorem, we arrive at \eqref{equality}.
\end{proof}

Denote by $\mathbf{B}$ and $\mathbf{B}'$ the Markov semigroups associated with $P_t(u,\cdot)$, which is the transition kernel defined by the process $u_t$. 
By the  Bogolyubov-Krylov argument (cf. \cite[Section 2.5.1]{kuk}), the existence of a stationary measure will be established if we show that the family $\{\overline{\lambda}_t, t\geq0\}$ is tight, where $\overline{\lambda}_t$ is defined as
\begin{equation}\label{lambda}
    \overline{\lambda}_t=\frac{1}{t}\int_0^t(\mathbf{B}'_s\lambda)(\Gamma)ds=\frac{1}{t}\int_0^t\int_{X}P_s(u,\Gamma)\lambda(du)ds,
\end{equation}
with $\Gamma\in B(X)$ is a Borel set and $\lambda$ is a probability measure on $\left( X,B(X) \right)$, $X$ Polish space. 
Since the Coriolis force does not contribute to the energy estimate \eqref{equality}, we prove the existence of a stationary measure as in \cite{kuk}. We include the proof here for the readers' convenience.
\begin{theorem}\label{theoremstationarymeasure}
Under the hypotheses of Theorem \ref{TheoSolution}, the stochastic Navier-Stokes system \eqref{Problemleray} with an arbitrary $\nu>0$ and $\alpha>0$ has a stationary measure.
\end{theorem}

\begin{proof}
Let us denote by $u(t,x)$ the solution of \eqref{Problemleray} issued from $u_0$ and by $\lambda_t$ the law of $u(t)$  as a random variable in $H$. To this end, it is enough to show that the family $\{\overline{\lambda}_t,t\geq0\}$ is tight, where $\overline{\lambda}_t$ is defined by relation \eqref{lambda}. Since the embedding $V\subset H$ is compact (Sobolev embedding Theorem), it is sufficient to prove that
\begin{equation}\label{limit}
    \sup_{t\geq0}\overline{\lambda}_t\left(H \backslash B_{V}(R)\right)\to 0,\ \text{ as }R\to \infty,
\end{equation}
where $B_{V}(R)$ stands for the ball in $V$ of radius $R$ centered in $0$. The relation \eqref{equality} with $u_0=0$ implies that
\begin{equation}\label{L2boundu}
    \alpha\mathbb{E}\int_0^t\norm{u}^2_{L^2}ds\leq \varepsilon t.
\end{equation}
Similarly we have 
\begin{equation}\label{L2boundgradu}
    \nu\mathbb{E}\int_0^t\norm{\nabla u}^2_{L^2}ds\leq \varepsilon t.
\end{equation}
By \eqref{L2boundu} and \eqref{L2boundgradu} we derive 
\begin{equation}\label{H1bound}
    \mathbb{E}\int_0^t\norm{u}^2_{H^1}ds\leq C\varepsilon t.
\end{equation}
Combining \eqref{H1bound} with Chebyshev's inequality, we arrive at
\begin{align*}
    \overline{\lambda}_t(H\backslash B_{V}(R))&=\frac{1}{t}\int_0^t \lambda_s(H\backslash B_{V}(R))ds=\frac{1}{t}\int_0^t\mathbb{P}\left\{\norm{u(s)}_{H^1}>R\right\}ds\\
    &\leq\frac{1}{t}R^{-2}\mathbb{E}\int_0^t \norm{u(s)}^2_{H^1}ds\leq C R^{-2}\varepsilon,
\end{align*}
whence follows convergence \eqref{limit}.

\end{proof}

\subsection{Regularity of stationary measure's support}\label{sec: regulaty solution}
\begin{theorem}\label{regularity solution thm}
In addition to the hypotheses of Theorem \ref{theoremstationarymeasure}, assume that 
\begin{equation}
    \varepsilon_1=\sum_{j=1}^\infty \alpha_jb_j^2<\infty
\end{equation}
where $\{\alpha_j\}$ are the eigenvalues of the Laplace operator $- \Delta u$
in $\mathbb{T}\cross I$. Let $\mu$, be a stationary measure of \eqref{Problem1}. Denotes the expectation w.r.t $\mu$ by $\mathbb{E}$. Then, we have: \begin{equation}\label{laplacianbound}
    \nu\mathbb{E}\norm{\Delta u}^2_{L^2}\leq\frac{\varepsilon_1}{2}.
\end{equation}
\end{theorem}

\begin{proof}
Assuming sufficiently regular $u$, we use the It\^{o} formula to the functional $\Phi(u)=-\langle \Delta u,u\rangle=\norm{\nabla u}^2_{L^2}$ (cf. \cite[Proposition 2.4.12]{kuk}). 

We observe,
\begin{equation*}
    \partial_t\Phi(u;v)=0, \ \ \
    \partial_u\Phi(u;v)=-2\langle\Delta u,v\rangle, \ \ \ \partial_u^2\Phi(u;v)=-2\langle \Delta v,v\rangle.
\end{equation*}
Then, we have 
\begin{align}\label{eqinter2}
    \norm{\nabla u(t\wedge\tau_n)}^2_{L^2}=\norm{\nabla u_0}^2_{L^2}&-2\int_0^{t\wedge\tau_n} \langle \Delta u,\nu \Delta u-\alpha u-(u\cdot \nabla)u-fu^\perp-\nabla p\rangle+\frac{\varepsilon_1}{2}ds \notag\\
    &-2\sum_{j=1}^\infty b_j\underbrace{\int_0^{t\wedge\tau_n} \langle \Delta u,e_j\rangle d\beta_j(s)}_{=:M(t\wedge\tau_n)},
\end{align}
where we set 
\begin{equation*}
    \tau_n=\inf\{t\geq0;\ \norm{\nabla u(t)}_{L^2}>n\}.
\end{equation*}
Using several integrations by parts and incompressibility, we get
\begin{align}
    \langle \Delta u, (u\cdot \nabla)u\rangle &=\langle \left(\partial_{11}u^1+\partial_{22}u^1,\partial_{11}u^2+\partial_{22}u^2\right),\left(u^1\partial_1u+u^2\partial_2u\right)\rangle=0,\\
    \langle \nabla p, \Delta u\rangle&=\sum_{j=1}^2\int_{\mathbb{T}\cross I}\Delta u^j\partial_j p dx=-\sum_{j=1}^2\int_{\mathbb{T}\cross I}p\Delta\partial_j u^j dx=0,\\
    \langle \Delta u, fu^\perp \rangle&=\int_{\mathbb{T}\cross I}-\left(\partial_{11} u^1+\partial_{22}u^1\right)\left(f_0+\beta x^2\right)u^2+\left(\partial_{11} u^2+\partial_{22}u^2\right)\left(f_0+\beta x^2\right)u^1dx\notag\\
    &=\beta\int_{\mathbb{T}\cross I}-\left(\partial_{11} u^1+\partial_{22}u^1\right) x^2u^2+\left(\partial_{11} u^2+\partial_{22}u^2\right) x^2u^1dx\notag\\
    &=\beta\int_{\mathbb{T}\cross I}u^2\partial_2 u^1-u^1\partial_2u^2 dx\notag\\    &=\beta\int_{\mathbb{T}\cross I}u^1\partial_1 u^1+\beta\int_{\mathbb{T}\cross I}u^2\partial_2u^2 dx=0.
\end{align}
Taking the mean value in \eqref{eqinter2} we obtain
\begin{align}
     &\mathbb{E}\norm{\nabla u(t\wedge\tau_n)}^2_{L^2}+2\alpha\mathbb{E}\int_0^{t\wedge\tau_n} \norm{\nabla u}^2_{L^2}\ ds+2\nu\mathbb{E}\int_0^{t\wedge\tau_n} \norm{\Delta u}^2_{L^2}\ ds\notag\\
     =&\mathbb{E}\norm{\nabla u_0}^2_{L^2}+\varepsilon_1\mathbb{E}(t\wedge\tau_n)
\end{align}
As seen in the proof of Proposition \ref{propenergybalance}, we can pass to the limit as $n\to\infty$ and we get
\begin{align}
     &\mathbb{E}\norm{\nabla u(t)}^2_{L^2}+2\alpha\mathbb{E}\int_0^{t} \norm{\nabla u}^2_{L^2}\ ds+2\nu\mathbb{E}\int_0^{t} \norm{\Delta u}^2_{L^2}\ ds
     =\mathbb{E}\norm{\nabla u_0}^2_{L^2}+\varepsilon_1t.
\end{align}
Moreover, if $u$ is a stationary solution, we have 
\begin{align}
     \alpha\mathbb{E} \norm{\nabla u}^2_{L^2}+\nu\mathbb{E} \norm{\Delta u}^2_{L^2}=\frac{\varepsilon_1}{2},
\end{align}
from which \eqref{laplacianbound}.
\end{proof}

\bibliographystyle{sn-basic}
\bibliography{sn-bibliography}

\end{document}